\documentclass{amsart}
\usepackage{amsfonts}
\usepackage{mathrsfs}
\usepackage{amssymb}
\usepackage{amsmath,amssymb}

\newtheorem{theorem}{Theorem}[section]
\newtheorem{lemma}[theorem]{Lemma}

\newtheorem{proposition}[theorem]{Proposition}
\newtheorem{corollary}[theorem]{Corollary}
\theoremstyle{definition}
\newtheorem{definition}[theorem]{Definition}

\theoremstyle{remark}
\newtheorem{remark}[theorem]{Remark}

\numberwithin{equation}{section}

\begin{document}

\title[Strichartz type estimates for fractional heat
equations] {Strichartz type estimates for fractional heat equations}
\author{Zhichun Zhai}
\address{Department of Mathematics and Statistics, Memorial University of Newfoundland, St. John's, NL A1C 5S7, Canada}
\curraddr{}
 \email{a64zz@mun.ca}
\thanks{Project supported in part  by Natural Science and
Engineering Research Council of Canada.}
\subjclass[2000]{Primary 35K05; 35K15; 35B65; 35Q30}
\keywords{Strichartz estimates; Time dependent potentials;
Fractional heat equations;
 Navier-Stokes equations}

\date{}

\begin{abstract}
  We obtain    Strichartz  estimates for the fractional  heat equations
  by using both the abstract Strichartz estimates of         Keel-Tao  and the
  Hardy-Littlewood-Sobolev inequality. We also prove an endpoint homogeneous
  Strichartz  estimate via replacing $ L^{\infty}_{x}(\mathbb{R}^{n})$
  by $BMO_{x}(\mathbb{R}^{n})$ and a parabolic homogeneous Strichartz
  estimate.  Meanwhile, we generalize the Strichartz estimates by replacing
  the  Lebesgue spaces with either Besov spaces or Sobolev spaces. Moreover,
  we establish the Strichartz estimates for the fractional  heat equations with
  a time dependent  potential of  an appropriate integrability.  As an application,
  we prove the global existence and uniqueness of  regular solutions in spatial
  variables for the generalized Navier-Stokes system with $L^{r}(\mathbb{R}^{n})$ data.
 \end{abstract}
\maketitle

 \vspace{0.1in} {\section{Introduction}}

 This paper studies  Strichartz type estimates for  the inhomogeneous initial
  problem associated with the fractional heat equations
\begin{equation}\label{eq1d}
\left\{\begin{array} {r@{\quad,\quad}l}
\partial_{t}v(t,x)+(-\triangle)^{\alpha}v(t,x)=F(t,x) & (t,x)\in\mathbb{R}^{1+n}_{+}=(0,\infty)\times\mathbb{R}^{n},\\
v(0,x)=f(x) & x\in \mathbb{R}^{n},
\end{array}
\right.
\end{equation}
 where   $\alpha\in(0, \infty)$ and $n\in\mathbb{N}.$  The main goal  is to determine
 pairs $(q,p)$ and $(q_{1}, p_{1})$ ensuring
\begin{equation}\label{eq1a}
\|e^{-t(-\triangle)^{\alpha}}f\|_{L^{q}_{t}(I; L^{p}_{x})}\lesssim \|f\|_{L^{2}},
\end{equation}
\begin{equation}\label{eq1b}
 \left\|\int_{0}^{t}e^{-(t-s)(-\triangle)^{\alpha}} F(s,x)ds\right\|_{L^{q}_{t}(I;L^{p}_{x})}
  \lesssim\|F\|_{L^{q_{1}'}_{t}(I;L^{p_{1}'}_{x})},
\end{equation}
where $I$ is either $[0, \infty)$ or $[0, T]$ for some $0<T<\infty,$
and $p_{1}'=\frac{p_{1}}{p_{1}-1}$ is the  conjugate  of a given
number $p_{1}\geq 1.$ Here    $\partial_{t}$ and $
\bigtriangleup=\sum_{j=1}^{n}\partial_{x_{j}}^{2} $ are the partial
derivative with respect to $t$ and the Laplacian with respect to
$x=(x_{1},\cdots,x_{n}),$ respectively. Furthermore,
$$
(-\triangle)^{\alpha}v(t,x)
 =\mathcal {F}^{-1}(|\xi|^{2\alpha}\mathcal{F}(v(t,\xi)))(x),
 $$
 where  $\mathcal {F}$  is  the Fourier  transform and  $\mathcal {F}^{-1}$ denotes its
  inverse. By the Fourier transform and Duhamel's principle, the solution of (\ref{eq1d})
    can be written as
$$
 v(t,x)=e^{-t(-\triangle)^{\alpha}}f(x)+\int_{0}^{t}e^{-(t-s)(-\triangle)^{\alpha}}F(s,x)ds,
$$
where
 $$
 e^{-t(-\triangle)^{\alpha}}f(x)=\mathcal
 {F}^{-1}(e^{-t|\xi|^{2\alpha}}\mathcal {F}f(\xi))(x)\nonumber
 :=K_{t}^{\alpha}(x)\ast f(x)
$$
 and $\ast$ stands for the convolution operating  on the space variable.

  The Strichartz  type estimates for equation (\ref{eq1d}) have just been  studied by few
  experts. Pierfelice \cite{V. Pierfelice} concerned such estimates   for  equation (\ref{eq1d})
  with $\alpha=1$ and  small potentials of very low regularity.  Miao, Yuan  and Zhang
  \cite{C. Miao} studied the non-endpoint  case of  (\ref{eq1a}) for equation (\ref{eq1d}).

  For the Schr$\ddot{o}$dinger and wave equations, the Strichartz estimates  have been well
  studied in recent years, see,   for example,  Blair-Smith-Sogge \cite{Blair Smith Sogge},
   Burq-G$\acute{e}$rard-Tzvetkov \cite{Burq gerard Tzvetkov}, Cazenave \cite{T. Cazenave},
   Kapitanski \cite{L. Kapitanski},   Keel-Tao \cite{M.Keel T.Tao}, Lindblad-Sogge
  \cite{H. Lindblad C. D. Sogge}, Mockerhaupt-Seeger-Sogge  \cite{G. Mockenhaupt A. Seeger C. D. Sogge },
  Staffilani-Tataru \cite{Staffilani Tataru}, Stefanov \cite{A. Stefanov}, Yajima-Zhang \cite{K. Yajima G. Zhang}.
  These estimates are very   important  in the study of  local and global existence for nonlinear equations,
  well posedness in Sobolev spaces with low order, scattering theory and many others, see, for example,
  Kenig-Merle \cite{Kenig Merle}, Kenig-Ponce-Vega \cite{Kenig Ponce Vega}, D'Ancona-Pierfelice-Visciglia
  \cite{P. D Ancona V. Pierfelice N. Visciglia}. The Strichartz  estimates for the Schr$\ddot{o}$dinger
   and wave equations can be  directly derived from the abstract Strichartz estimates of Keel-Tao \cite{M.Keel T.Tao}
   since  the solution groups of these two equations act as unitary operators on $L^{2}(\mathbb{R}^{n})$ and such
   operators obey both the energy estimate and the  untruncated decay estimate. While, since
   $\{e^{-t(-\triangle)^{\alpha}}\}_{t\geq 0}$ is a semigroup and acts as a self-adjoint operator on $L^{2}(\mathbb{R}^{n})-$see
   Lemma \ref{2a}, we can only apply the abstract Strichartz estimates of Keel-Tao directly to obtain (\ref{eq1a}) if
   we have the energy estimate and untruncated decay estimate.  But for (\ref{eq1b}),  we  can  make use of the
   $L^{p}-$decay estimates and the Hardy-Littlewood-Sobolev inequality.

  In this paper, we  also  establish an endpoint case of  (\ref{eq1a}) by  replacing $L_{x}^{\infty}(\mathbb{R}^{n})$
  with  the  spaces of functions of bounded mean  oscillation $(BMO_{x}({\mathbb{R}^{n}}))$.  Meanwhile, we obtain a
  parabolic homogeneous  Strichartz estimate for equation (\ref{eq1d}), the two dimensional case  of which is very useful
  for dealing with the global regularity of wave maps when combined with Lemma \ref{2b} for $\alpha=1$ and the comparison
  principle  for the heat equation, see Tao \cite{T. Tao 1}. Moreover, we generalize (\ref{eq1a}) and (\ref{eq1b}) via replacing
  $L^{p}(\mathbb{R}^{n})$ with either Besov spaces or   Sobolev spaces. These function  spaces will be made precise later.

 If equation (\ref{eq1d}) has a time dependent potential $V(t,x),$ then it becomes
  \begin{equation}\label{eq1111d}
\left\{\begin{array} {r@{\quad,\quad}l}
\partial_{t}v(t,x)+(-\triangle)^{\alpha}v(t,x)+V(t,x)v(t,x)=F(t,x) & (t,x)\in\mathbb{R}^{1+n}_{+},\\
v(0,x)=f(x) & x\in \mathbb{R}^{n}.
\end{array}
\right.
\end{equation}
 We can  obtain the Strichartz estimates for equation (\ref{eq1111d}) by using  the Banach contraction mapping principle
 and assuming an appropriate integrability condition in space and time on $V(t,x).$ A similar idea was  used  by
 D'Ancona-Pierfelice-Visciglia in \cite{P. D Ancona V. Pierfelice N. Visciglia} to get analogous estimates for
 the Schr$\ddot{o}$dinger  equations.

 As an application, we establish the global existence and uniqueness of regular solutions in spatial variables for
 the generalized Navier-Stokes system on the half-space $\mathbb{R}^{1+n}_{+},$ $n\geq 2:$
\begin{equation}\label{eq1dd}
 \left\{\begin{array} {r@{\quad\quad}l}
 \partial_{t}v+(-\triangle)^{\alpha}v+(v\cdot \nabla)v-\nabla p=h,
 & \hbox{in}\ \mathbb{R}^{1+n}_{+};\\
 \nabla \cdot v=0, & \hbox{in}\ \mathbb{R}^{1+n}_{+};\\
 v(0,x)=g(x), &\hbox{in}\ \mathbb{R}^{n} \end{array} \right.
\end{equation}
 with $\alpha\in(\frac{1}{2},\frac{1}{2}+\frac{n}{4}).$ For system (\ref{eq1dd}), Lions  \cite{J L Lions} proved the
 global  existence of  the classical solutions when $\alpha\geq \frac{5}{4}$ in dimensional $3.$
 Similar result holds  for general  dimension $n$ if $\alpha\geq \frac{1}{2}+\frac{n}{4},$ see Wu \cite{J Wu 1} and
 \cite{J Wu 2}. The mild  solutions for system (\ref{eq1dd}) are
 $$
 v(t,x)=e^{-t(-\triangle)^{\alpha}}g(x)+\int_{0}^{t}e^{-(t-s)(-\triangle)^{\alpha}}P(h-\nabla(v\otimes v))ds,
 $$
 where
 $P$ is the Helmboltz-Weyl projection:
 $$
 P=\{P_{j,k}\}_{j,k=1,\cdots,n}=\{\delta_{j,k}+R_{j}R_{k}\}_{j,k=1,\cdots,n}
 $$
 with $\delta_{j,k}$ being the Kronecker symbol and $R_{j}=\partial_{j}(-\triangle)^{-1/2}$ being the Riesz transform.
 When $\alpha=1,$ system  (\ref{eq1dd}) becomes the classical  Navier-Stokes system which is  a celebrated nonlinear
 partial differential system.

 In the above and below,  $U\lesssim V$ denotes $U\leq C V$ for  some positive constant $C$ which is independent
 of the sets or functions under consideration in both $U$ and $V;$ for  a Banach space $X$,  $L^{p}(X)$
 (where $p\in[1,\infty)$) is used as  the space of functions $f:X\longrightarrow \mathbb{R}$ with
$$
 \|f\|_{L^{p}(X)}=\left(\int_{X}|f(x)|^{p}dx\right)^{1/p}<\infty;
$$
 for a function space $F(\mathbb{R}^{n})$ on $\mathbb{R}^{n},$ $L^{q}(I;F(\mathbb{R}^{n}))$ (where $q \in[1,\infty)$)
 represents the set of  functions $f:I\times \mathbb{R}^{n}\longrightarrow \mathbb{R}$ for $I\subseteq \mathbb{R}$ with
$$
\|f\|_{L^{q}(I;
F(\mathbb{R}^{n}))}=\left(\int_{I}\|f(t,x)\|_{F(\mathbb{R}^{n})}^{q}dt\right)^{1/q}<\infty.
$$

 To state our main results,  let us recall the definitions of some function spaces.

 We use $\mathscr{S}_{0}$ to denote the following subset of the Schwartz class of rapidly decreasing functions $\mathscr{S},$
 $$
 \mathscr{S}_{0}=\left\{\phi\in \mathscr{S}: \int_{\mathbb{R}^{n}}\psi(x)x^{\gamma}dx,|\gamma|=0,1,2,\cdots\right\},
 $$
 where $x^{\gamma}=x_{1}^{\gamma^{1}}x_{2}^{\gamma^{2}} \cdots x_{n}^{\gamma^{n}},|\gamma|=\gamma_{1}+\gamma_{2}+\cdots+\gamma_{n}.$
 Its dual$\mathscr{S}'_{0}=\mathscr{S}'/\mathscr{S}^{\bot}_{0}=\mathscr{S}'/\mathcal{P},$  where $\mathcal {P}$ is the space of multinomials.

 We introduce a dyadic partition of $\mathbb{R}^{n}.$ For each $j\in\mathbb{Z},$ we let
 $$
  D_{j}=\{\xi: 2^{j-1}<|\xi|\leq 2^{j+1}\}.
 $$
  We choose $\phi_{0}\in\mathscr{S}(\mathbb{R}^{n})$ such that
 $$
  \hbox{supp}(\phi_{0})=\{\xi:2^{-1}\leq |\xi|\leq 2\}\ \hbox{and}\
  \phi_{0}>0 \ \hbox{on}\ D_{0}.
 $$
 Let
$$
 \phi_{j}(\xi)=\phi_{0}(2^{-j}\xi)\ \hbox{and}\ \widehat{\Psi_{j}}(\xi)=\frac{\phi_{j}(\xi)}{\sum_{j}\phi_{j}(\xi)}.
$$
Then $\Psi_{j}\in\mathscr{S}$ and
$$
\widehat{\Psi_{j}}(\xi)=\widehat{\Psi_{0}}(2^{-j}\xi), \ \hbox{supp}(\widehat{\Psi_{j}})\subset D_{j},
 \Psi_{j}(x)=2^{jn}\Psi_{0}(2^{j}x).
 $$
 Moreover,
 \begin{equation}\label{eq2e}
\sum_{k=-\infty}^{\infty}\widehat{\Psi_{k}}(\xi)=\left\{\begin{array}
{l@{\quad \quad}l}
 1,\ \hbox{if}\ \xi\in \mathbb{R}^{n}\backslash \{0\}, \\
 0,\ \hbox{if}\ \xi=0.
 \end{array}
  \right.
\end{equation}

 Let $\Phi\in C_{0}^{\infty}(\mathbb{R}^{n})$ be even and satisfy
 $$
 \widehat{\Phi}(\xi)=1-\sum_{k=0}^{\infty}\widehat{\Psi}_{k}(\xi).
 $$
 Then, for any $\psi\in\mathscr{S},$
$$
 \Phi\ast\psi+\sum_{0}^{\infty}\Psi_{k}\ast\psi=\psi
$$
 and for any $f\in \mathscr{S},$
$$
 \Phi\ast f+\sum_{k=0}^{\infty}\Psi_{k}\ast f=f.
$$

 To define the homogeneous Besov spaces, we let
$$
\triangle_{j}f=\Psi_{j}\ast f, j=0,\pm1,\pm2,\cdots.
$$
 For $s\in \mathbb{R}^{n}$ and $1\leq p,q\leq \infty,$ we define the homogeneous Besov space
 $\dot{B}^{s}_{p,q}$ as the set of all $f\in \mathscr{S}'_{0}$ with
$$
\|f\|_{\dot{B}^{s}_{p,q}}=\left(\sum_{j=-\infty}^{\infty}(2^{js}\|\triangle_{j}f\|_{L^{p}})^{q}\right)^{1/q}<\infty, \ \hbox{for}\ q<\infty,
$$
$$
\|f\|_{\dot{B}^{s}_{p,q}}=\sup_{-\infty<j<\infty}2^{js}\|\triangle_{j}f\|_{L^{p}}<\infty, \ \hbox{for}\ q=\infty.
$$
 To define the inhomogeneous Besov spaces, we define
 \begin{equation}\label{eq2ee}
 \triangle_{j}f=\left\{\begin{array} {l@{\quad \quad}l}
 0,\ \ \ \ \ \ \ \ \ \ \hbox{if}\ j \leq -2, \\
 \Phi\ast f,\ \ \ \hbox{if}\ j=-1,\\
 \Psi_{j}\ast f, \ \ \hbox{if}\ j=0,1,2,\cdots
 \end{array}
  \right.
\end{equation}
 For $s\in \mathbb{R}^{n}$ and $1\leq p,q\leq \infty,$ we define the inhomogeneous Besov space $B^{s}_{p,q}$ as the set of all $f\in
\mathscr{S}'$ with
$$
\|f\|_{B^{s}_{p,q}}=\|\triangle_{-1}f\|_{L^{p}}+\left(\sum_{j=0}^{\infty}(2^{js}\|\triangle_{j}f\|_{L^{p}})^{q}\right)^{1/q}<\infty, \ \hbox{for}\ q<\infty,
$$
$$
\|f\|_{B^{s}_{p,q}}=\|\triangle_{-1}f\|_{L^{p}}+\sup_{0\leq j<\infty}2^{js}\|\triangle_{j}f\|_{L^{p}}<\infty, \ \hbox{for}\ q=\infty.
$$

 On the  other hand, Besov spaces  can be defined by interpolation between the Lebesgue spaces and the   Sobolev spaces of integer
 order (see  Triebel \cite{H. Triebel}). Moreover, it follows from Bergh and L$\ddot{o}$fstr$\ddot{o}$m \cite{Berg Lofstrom}  that
 for $s\in \mathbb{R}$ and $1\leq p, q\leq \infty,$
$$
 B^{s}_{p,q}(\mathbb{R}^{n})=[H^{s_{1},p},H^{s_{2},p}]_{\theta, q}\ \hbox{and}\
 \dot{B}^{s}_{p,q}(\mathbb{R}^{n})=[\dot{H}^{s_{1},p},\dot{H}^{s_{2},p}]_{\theta, q},
$$
 where $s_{1}\neq s_{2},$ $0<\theta<1$ and $s=(1-\theta)s_{1}+\theta s_{2}.$  Here $H^{s,p}(\mathbb{R}^{n})$  and
 $\dot{H}^{s,p}(\mathbb{R}^{n})$ are the inhomogeneous and homogeneous Sobolev spaces which are the completion  of
 all infinitely differential functions $f$ with compact support in $\mathbb{R}^{n}$ with respect to the norms
$$
\|f\|_{{H}^{s,p}(\mathbb{R}^{n})}=\|(I-\triangle)^{s/2}f\|_{L^{p}(\mathbb{R}^{n})},
\hbox{and}\
\|f\|_{\dot{H}^{s,p}(\mathbb{R}^{n})}=\|(-\triangle)^{s/2}f\|_{L^{p}(\mathbb{R}^{n})}\
$$
 respectively, where $(I-\triangle)^{s/2}f=\mathcal {F}^{-1}((1+|\xi|^{2})^{s/2}\mathcal {F}f(\xi)).$

  $BMO(\mathbb{R}^{n})$ is the set of  locally integrable functions $f$ with semi-norm
$$
\|f\|_{BMO}=\left(\sup\limits_{Q}\mathcal
{L}(Q)^{-n}\int_{Q}|f(x)-f_{Q}|^{2}dx\right)^{1/2}<\infty,
$$
 where $Q$ is a cube in $\mathbb{R}^{n}$ with sides parallel to the coordinate axes, $\mathcal {L}(Q)$ is the
 sidelength of $Q$ and $f_{Q}=\mathcal {L}(Q)^{-n}\int_{Q}f(x)dx.$

\begin{definition}
  The triplet $(q,p,r)$ is called a $\sigma-$admissible triplet provided
$$
\frac{1}{q}=\sigma\left(\frac{1}{r}-\frac{1}{p}\right),
$$
where $1<r\leq p\leq\infty$ and $\sigma>0.$
\end{definition}

\begin{proposition} \label{a2}
Let $(q,p,2)$ be $\frac{n}{2\alpha}-$admissible. If $q\geq 2$ and $(q,p,\frac{n}{2\alpha})$ is not $(2,\infty, 1),$  then
(\ref{eq1a}) holds.
\end{proposition}

\begin{remark} \label{a3}   Proposition \ref{a2} extends  Miao-Yuan-Zhang's \cite[Lemma 3.2]{C. Miao}  to the cases:
 $(q, p, r)=(2,\frac{2n}{n-2\alpha},2)$ when $n>2\alpha;$ $(q,p,r)=(\frac{4\alpha}{n},\infty,2)$ when $n<2\alpha.$
\end{remark}

 It is well known that for the Schr$\ddot{o}$dinger equations,  there are  pairs  $(q,p)$ and $(q_{1},p_{1})$ such that  $(q,p,2)$
 and $(q_{1},p_{1},2)$   are not $n/2-$admissible but   the inhomogeneous  Strichartz estimates hold (see Cazenave-Weissler
 \cite{T. Cazenave F. B. Weissler},  Kato \cite{T. Kato} and Vilela \cite{M. C. Vilela}). Similarly, we will prove that (\ref{eq1b})
  holds  for   some pairs $(q,p)$ and $(q_{1},p_{1})$ satisfying the property
\begin{equation}\label{eq1j}
\left(\frac{1}{q_{1}'}-\frac{1}{q}\right)+\frac{n}{2\alpha}\left(\frac{1}{p_{1}'}-\frac{1}{p}\right)=1.
\end{equation}
 This property is weaker than  the $\frac{n}{2\alpha}-$admissibility  of $(q,p,2)$ and $(q_{1}, p_{1}, 2).$

\begin{theorem}\label{a5}
Let $1\leq p_{1}'<p\leq \infty $  and $1< q_{1}'< q<\infty .$ If $(q,p)$ and $(q_{1},p_{1})$ satisfy  (\ref{eq1j}), then (\ref{eq1b}) holds.
\end{theorem}

\begin{remark}\label{a6} Since $e^{-t(-\triangle)^{\alpha}}$ commutes with  $(-\triangle)^{\beta}$ and $(I-\triangle)^{\beta}$  for
 $\beta>0,$ if $(q,p)$ satisfies the assumption of Theorem \ref{a2} then (\ref{eq1a}) holds with $\|\cdot\|_{L^{p}(\mathbb{R}^{n})}$
 replaced  by  either $\|\cdot\|_{\dot{H}^{\beta, p}(\mathbb{R}^{n})}$ or $\|\cdot\|_{H^{\beta, p}(\mathbb{R}^{n})}$.
 Similarly, if $(q,p)$ and $(q_{1},p_{1})$ satisfy the assumption of Theorem \ref{a5},  then (\ref{eq1b})  holds with the same replacement.
\end{remark}

\begin{theorem} \label{a4}Let $n=2\alpha.$ Then \\
\begin{equation}\label{eq1i}
\|e^{-t(-\triangle)^{\alpha}}f\|_{L^{2}_{t}((0, \infty);BMO_{x}(\mathbb{R}^{n}))}\lesssim\|f\|_{L^{2}}.
\end{equation}
\end{theorem}

\begin{theorem}\label{111111}
(a) Let $1\leq r\leq p\leq \infty$ and $0<T<\infty.$ If $n<2\alpha,$ then
\begin{equation}\label{222222}
\int_{0}^{T}s^{-\frac{nr}{2p\alpha}}\|e^{-s(-\triangle)^{\alpha}}f\|_{L^{p}_{x}(\mathbb{R}^{n})}^{r}ds
\lesssim T^{1-\frac{n}{2\alpha}}\|f\|^{r}_{L^{r}(\mathbb{R}^{n})}.
\end{equation}
(b) Let $2< p\leq \infty.$ If $n=2\alpha,$ then
\begin{equation}\label{333333}
\int_{0}^{\infty}s^{-2/p}\|e^{-s(-\triangle)^{\alpha}}f\|_{L^{p}_{x}(\mathbb{R}^{n})}^{2}ds
\lesssim \|f\|^{2}_{L^{2}(\mathbb{R}^{n})}.
\end{equation}

\end{theorem}
\begin{remark} We can refer to  (\ref{333333})  as  a parabolic homogeneous Strichartz estimate.  The special case  $n=2$ of
 (\ref{333333})   was  proved by Tao in \cite{T. Tao 1}. On the other hand,  according to Miao-Yuan-Zhang's \cite[Proposition
 2.1]{C. Miao}, (\ref{333333}) amounts to the fact that $L^{2}(\mathbb{R}^{n})$  is embedded in the homogeneous Besov space
 $$
 \dot{B}_{p,2}^{s}(\mathbb{R}^{n}),\ \  s=\frac{(2-p)n}{2p},\ \  2<p\leq \infty.
 $$
\end{remark}

 Using  the imbedding of $\dot{H}^{\alpha,2}$  into $L^{\frac{2n}{n-2\alpha}}$ when $0<2\alpha<n,$ we prove the
following result.

 \begin{theorem}\label{a8} Let $n>2\alpha>0,$  $p\in[1, 2),$ $q \in (1, 2).$
 If  $\frac{1}{q}+\frac{n}{2\alpha}\left(\frac{1}{p}-\frac{1}{2}\right)=\frac{3}{2}$
then
\begin{equation}\label{eq1s}
\left\|\int_{0}^{t}e^{-(t-s)(-\triangle)^{\alpha}}F(s,x)ds\right
\|_{L^{2}_{t}(I;L^{\frac{2n}{n-2\alpha}}_{x})}
\lesssim\|F\|_{L^{q}_{t}(I;Z)}
\end{equation}
 holds with  $Z=\dot{H}^{\alpha, p}_{x}$ or ${H}^{\alpha, p}_{x}.$
\end{theorem}

 Using the Littlewood-Paley decomposition, we  establish the following estimates in  Besov spaces.
\begin{corollary}\label{a7}
(a) Let $(q,p,2)$ be $\frac{n}{2\alpha}-$admissible. If $q\geq 2$  and $(q,p,\frac{n}{2\alpha})$ is not $(2,\infty, 1),$ then \\
\begin{equation}\label{eq1k}
\|e^{-t(-\triangle)^{\alpha}}f\|_{L^{q}_{t}(I; X_{1})}\lesssim
\|f\|_{X_{2}}
\end{equation}
 holds with $(X_{1},X_{2})=(B^{s}_{p, 2}, {B}^{s}_{2, 2})$ or $(\dot{B}^{s}_{p, 2}, \dot{B}^{s}_{2, 2}).$\\
 (b) Let $1\leq p_{1}'<p\leq \infty $  and $1< q_{1}'< q<\infty .$ If $(q,p)$ and $(q_{1},p_{1})$ satisfy  (\ref{eq1j})
 and $q_{1}\geq 2,$  then
\begin{equation}\label{eq1m}
\left\|\int_{0}^{t}e^{-(t-s)(-\triangle)^{\alpha}}F(s, x)ds\right\|_{L^{q}_{t}(I;Y_{1})} \lesssim\|F\|_{L^{q_{1}'}_{t}(I;Y_{2})}
\end{equation}
 holds with  $(Y_{1},Y_{2})=(B^{s}_{p,2},B^{s}_{ p_{1}',2})$ or $(\dot{B}^{s}_{p,2}, \dot{B}^{s}_{ p_{1}', 2}).$
\end{corollary}

\begin{corollary}\label{88}
 Let $n\geq 2\alpha,$ $I=[0,T]$ or $[0,\infty).$   Suppose $V$ is a real potential and
 $$
 V\in L^{r}_{t}(I; L^{s}_x),\  \frac{1}{r}+\frac{n}{2\alpha s}=1,
 $$
  for some fixed $r\in(1,2)\cup(2,\infty)$ and $s\in(\frac{n}{2\alpha}, \frac{n}{\alpha})\cup(\frac{n}{\alpha},\infty).$ Let
  $f\in L^{2}$ and $F\in L^{q_{1}'}_{t}(I; L^{p_{1}'}_{x})$ for some $\frac{n}{2\alpha}-$admissible triplet $(q_{1}, p_{1}, 2)$
   with  $p_{1}'\in[1,2)$ and $q_{1}'\in(1,2).$ Then equation (\ref{eq1111d}) has a unique solution  $v(t,x)$ satisfying
\begin{equation}\label{1111c}
 \|v\|_{L^{q}_{t}(I; L^{p}_{x})}\lesssim \|f\|_{L^{2}}+\|F\|_{L^{q_{1}'}_{t}(I; L^{p_{1}'}_{x})},
\end{equation}
  for all   $\frac{n}{2\alpha}-$admissible triplets $(q, p, 2)$ with $2\leq q<\infty.$
\end{corollary}

 We can prove the following estimate by estimating $K_{t}^{\alpha}(x)$ in mixed norm spaces.
\begin{theorem}\label{a9} Let $\alpha>0,$ $0<T< \infty,$ $1\leq p_{1}'<p\leq \infty , $
 $1\leq q_{1}'<q\leq \infty,$
 $\frac{1}{r}=\frac{1}{p}+\frac{1}{p_{1}}$
 and
$\frac{1}{h}=\frac{1}{q}+\frac{1}{q_{1}}.$
 If
$$
0<\frac{nh}{2\alpha}\left(1-\frac{1}{r}\right)< 1,
$$
 then
\begin{equation}\label{eq1u}
\left\|\int_{0}^{t}e^{-(t-s)(-\triangle)^{\alpha}}F(s,x)ds\right
\|_{L^{q}_{t}([0, T); X)} \lesssim
T^{\frac{1}{h}-\frac{n}{2\alpha}(1-\frac{1}{r})
}\|F\|_{L^{q_{1}'}_{t}([0, T); Y)}
\end{equation}
 holds with $(X,Y)=(L^{p}_{x},L^{p'_{1}}_{x}),$
 $(\dot{H}^{\beta,p}_{x},\dot{H}^{\beta,p'_{1}}_{x})$ or
 $(H^{\beta,p}_{x},H^{\beta,p'_{1}}_{x})$ for all $\beta>0.$
\end{theorem}

 In the rest of this paper, we use the notation $L^{p}$ indiscriminately for scalar and vector valued functions.

\begin{proposition}\label{propostion} Let $\alpha>1/2$ and $T>0.$ Assume that $u,v\in
 L^{q}([0,T];L^{p})$ with $p,q$ satisfying
$$
\max\left\{\frac{n}{2\alpha-1},2\right\}<p<\infty,\ \  2\alpha-1=\frac{2\alpha}{q}+\frac{n}{p}.
$$
 Then  the operator
 $$
 B(u,v)=\int^{t}_{0}e^{-(t-s)(-\triangle)^{\beta}}P\nabla\cdot(u\otimes v)ds
 $$
 is bounded from $L^{q}([0,T];L^{p})\times L^{q}([0,T];L^{p})$ to
 $L^{q}([0,T];L^{p})$ with
$$
 \|B(u,v)\|_{L^{q}([0,T];L^{p})}\lesssim\|u\|_{L^{q}([0,T];L^{p})}\|v\|_{L^{q}([0,T];L^{p})}.
$$
\end{proposition}

  Applying Theorems \ref{a5} \& \ref{a9},   Proposition \ref{propostion} and  Lemma  \ref{2 lemma 3}, we
  obtain the global existence and uniqueness of solutions for  system (\ref{eq1dd}).

 \begin{proposition}\label{10} Let $\alpha\in(\frac{1}{2}, \frac{1}{2}+\frac{n}{4}),$ $0<T<\infty,$ $p>\frac{n}{2\alpha-1}$ and
 $\frac{n}{p}+\frac{2\alpha}{q}=2\alpha-1.$\\
(a) Assume that   $\frac{n}{2\alpha-1}< r\leq p,$  $1\leq p_{1}'
<p<\infty,$  $1\leq q_{1}' <q\leq \infty,$
$$
0<\frac{n}{2\alpha}\left(\frac{1}{q}+\frac{1}{q_{1}}\right)\left(1-\frac{1}{p}-\frac{1}{p_{1}}\right)<1,
$$
 $g\in L^{r}(\mathbb{R}^{n})$ with $\nabla\cdot g=0$ and $h\in L^{q_{1}'}_{t}([0, T]; L^{p_{1}'}_{x}(\mathbb{R}^{n})).$ If
 there exists a suitable   constant $C>0$  such that
\begin{equation}\label{assum co}
T^{1-\frac{n}{2\alpha}\left(\frac{1}{n}+\frac{1}{r}\right)}\|g\|_{L^{r}(\mathbb{R}^{n})}
 +T^{\frac{1}{q}+\frac{1}{q_{1}}-\frac{n}{2\alpha}\left(\frac{1}{p_{1}'}-\frac{1}{p}\right)}
 \|h\|_{L^{q_{1}'}_{t}([0,T];
 L^{p_{1}'}_{x}(\mathbb{R}^{n}))}\leq C,
\end{equation}
 then (\ref{eq1dd}) has a unique strong solution $v\in L^{q}_{t}([0, T]; L^{p}_{x} (\mathbb{R}^{n}))$ in the sense of
$$
v=e^{-t(-\triangle)^{\alpha}}g(x)+
\int_{0}^{t}e^{-(t-s)(-\triangle)^{\alpha}}P[h(s,x)-\nabla\cdot(v\otimes
v)(s,x)]ds,
$$
 (b) Assume that $g\in L^{\frac{n}{2\alpha-1}}(\mathbb{R}^{n})$ with $\nabla\cdot g=0$ and $h\in L^{q_{1}'}_{t}([0,\infty);
 L^{p_{1}'}_{x}(\mathbb{R}^{n}))$ with $q_{1}'$ and $p_{1}'$ satisfying $1< q_{1}' <q<\infty,$
$$
1\leq p_{1}' <p<\left \{\begin{array}{ll}
\frac{n^{2}}{(n-2\alpha)(2\alpha-1)}, & 2\alpha<n, \\
\infty, & 2\alpha\geq n, \end{array} \right. \ \hbox{and}\
\frac{n}{p'_{1}}+\frac{2\alpha}{q'_{1}}=4\alpha-1.
$$
 If  $\|g\|_{L^{\frac{n}{2\alpha-1}}(\mathbb{R}^{n})}+\|h\|_{L^{q_{1}'}_{t}([0,\infty); L^{p_{1}'}_{x}(\mathbb{R}^{n}))} $
 is small enough, then (\ref{eq1dd}) has a unique strong  solution $v\in L^{q}_{t}([0, \infty);L^{p}_{x} (\mathbb{R}^{n})).$
\end{proposition}

 We show that the solution established in  proposition \ref{10} is smooth in spatial variables.  For a non-negative multi-index
$k=(k_{1},\cdots,k_{n})$ we define
$$
 D^{k}=\left(\frac{\partial}{\partial_{x_{1}}}\right)^{k_{1}}\cdots\left(\frac{\partial}{\partial_{x_{n}}}\right)^{k_{n}}
$$
 and $|k|=k_{1}+\cdots+k_{n}.$

\begin{corollary}\label{proposition}
 Under the hypothesis  of Corollary \ref{10} we  assume further  that for a
 non-negative multi-index  $k$
 $$
  D^{k}g\in L^{r}\ \hbox{and}\ D^{k}h\in L^{q_{1}'}([0,T];L^{p_{1}'}).
 $$
 Then  the solution $v$ established in Corollary \ref{10} satisfies
 \begin{equation}
 D^{j}v\in L^{q}([0,T];L^{p}),
 \end{equation}
 for any non-negative multi-index $j$ with $|j|\leq |k|.$
\end{corollary}

  The rest of this paper is organized as follows. In the next section,  we  give some basic lemmas: Lemma \ref{2a} states
  that $e^{-t(-\triangle)^{\alpha}}$  commutes with fractional  derivatives and  is self-adjoint as an operator on
  $L^{2}(\mathbb{R}^{n});$ Lemmas \ref{2b}-\ref{2c} provide us the $L^{p}-$decay estimates and non-endpoint strichartz
  estimates  of  the fractional  heat equation  established by Miao-Yuan-Zhang  in \cite{C. Miao};  Lemma \ref{2 lemma 3}
  gives another mixed norm estimate of  $e^{-t(-\triangle)^{\alpha}}f;$ Lemma \ref{2d} is the well known abstract
  Strichartz estimates of Keel-Tao \cite{M.Keel T.Tao}. In the third section,  we prove the main results of this paper:
  Proposition  \ref{a2} is derived  from the abstract Strichartz estimates. Theorem \ref{a5} is  proved   by the
  Hardy-Litllewood-Sobolev inequality and Lemmas \ref{2a} \& \ref{2b}.  Theorem  \ref{a4} is verified by the
  Littlewood-Paley decomposition. (a) of  Theorem \ref{111111} is derived from Lemma \ref{2b}  and the proof of (b) is
  essentially  the same as the proof of  Tao's \cite[Lemma 2.5]{T. Tao 1}. Theorem \ref{a8} is demonstrated according to
  the imbedding of $\dot{H}^{\alpha, 2}(\mathbb{R}^{n})$ into $L^{\frac{2n}{n-2\alpha}}(\mathbb{R}^{n})$ when
  $\alpha\in (0,\frac{n}{2})$ and the Hardy-Littlewood-Sobolev  inequality. Corollary \ref{a7} is showed by Proposition
  \ref{a2}  and Theorem  \ref{a5} via the definition  of Besov spaces.  Corollary \ref{88} is proved  form  Proposition
  \ref{a2},  Theorem \ref{a5} and the Banach contraction mapping principle.  Theorem  \ref{a9} is showed by using  the
  Young inequality and estimating $K_{t}^{\alpha}(x)$ in  mixed norm spaces. Proposition \ref{propostion} is proved via
  Lemma \ref{2b} and  the Hardy-Littlewood-Sobolev inequality.   Proposition \ref{10} is established by applying Proposition
  \ref{propostion},  Lemma  \ref{2 lemma 3}, the Banach contraction mapping principle and our main results. Corollary
  \ref{proposition} is verified by induction and the Banach contraction mapping principle.

\vspace{0.1in}
{\section{Lemmas}}

 This section contains five results needed for proving  the main results of this paper. The first one states that
 $e^{-t(-\triangle)^{\alpha}}$  commutes with $(-\triangle)^{\beta}$ and $(I-\triangle)^{\beta},$ and it is a self-adjoint
 bounded linear operator on $L^{2}(\mathbb{R}^{n}).$

\begin{lemma}\label{2a} For all $t>0$ and $\beta,\alpha>0,$ we have\\
(a) $e^{-t(-\triangle)^{\alpha}}(-\triangle)^{\beta}=(-\triangle)^{\beta}e^{-t(-\triangle)^{\alpha}}.$\\
(b) $e^{-t(-\triangle)^{\alpha}}(I-\triangle)^{\beta}=(I-\triangle)^{\beta}e^{-t(-\triangle)^{\alpha}}.$\\
(c) $\langle e^{-t(-\triangle)^{\alpha}}f, g\rangle =\langle f,
e^{-t(-\triangle)^{\alpha}}f\rangle, \forall f, g\in
L^{2}(\mathbb{R}^{n}).$
\end{lemma}
\begin{proof}  The proofs of (a) and (b) will  follow form  the definition of $e^{-t(-\triangle)^{\alpha}},$
$(-\triangle)^{\beta}$ and $(I-\triangle)^{\beta}.$ For (b),  let $f, g\in L^{2}(\mathbb{R}^{n}).$ According
 to the  Fourier transform and the Plancherel's identity we have
\begin{eqnarray}
\langle e^{-t(-\triangle)^{\alpha}}f, g\rangle
&=&\int (e^{-t(-\triangle)^{\alpha}}f) \overline{g(x)}dx\nonumber\\
&=&\int\mathcal {F}^{-1}\left(e^{-t|\xi|^{2\alpha}}\mathcal {F}f(\xi)\right)(x)\overline{g(x)}dx
\nonumber\\
&=&\int e^{-t|\xi|^{2\alpha}}\mathcal {F}f(\xi)\overline{\mathcal{F}g(\xi)}d\xi
\nonumber\\
&=&\int \mathcal {F}f(\xi)\overline{e^{-t|\xi|^{2\alpha}}\mathcal{F}g(\xi)}d\xi
\nonumber\\
&=&\int f(x)\overline{\mathcal{F}^{-1}(e^{-t|\xi|^{2\alpha}}\mathcal{F}g(\xi))(x)}dx\nonumber\\
&=&\langle f, e^{-t(-\triangle)^{\alpha}}g\rangle.\nonumber
\end{eqnarray}
 This finishes the proof of Lemma \ref{2a}.
 \end{proof}
  Miao-Yuan-Zhang in \cite{C. Miao} established the forthcoming two lemmas.
\begin{lemma}\label{2b} \cite{C. Miao}
  Let $1\leq r\leq p\leq \infty$ and $f\in L^{r}(\mathbb{R}^{n}).$ Then
 $$
 \|e^{-t(-\triangle)^{\alpha}}f\|_{L^{p}_{x}}\lesssim t^{-\frac{n}{2\alpha}\left(\frac{1}{r}-\frac{1}{p}\right)}\|f\|_{L^{r}},
 $$
$$
 \|\nabla e^{-t(-\triangle)^{\alpha}}f\|_{L^{p}_{x}}\lesssim t^{-\frac{1}{2\alpha}-\frac{n}{2\alpha}\left(\frac{1}{r}-\frac{1}{p}\right)}\|f\|_{L^{r}}.
 $$
\end{lemma}

\begin{lemma}\label{2c} \cite{C. Miao}
 Let $(q,p,r)$ be any  $\frac{n}{2\alpha}-$admissible triplet satisfying
 $$
 p<\left \{\begin{array}{ll}
\frac{nr}{n-2\alpha}, & n>2\alpha, \\
\infty, & n\leq 2\alpha, \end{array} \right.
$$
 and let $\varphi\in L^{r}(\mathbb{R}^{n}).$ Then $ e^{-t(-\triangle)^{\alpha}}\varphi\in L^{q}(I; L^{p}(\mathbb{R}^{n}))$ with the estimate
$$
\|e^{-t(-\triangle)^{\alpha}}\varphi\|_{L^{q}_{t}(I; L^{p}_{x})}\lesssim\|\varphi\|_{r},
$$
for $I=[0, T),$ $0<T\leq\infty.$
\end{lemma}

We can obtain the following estimate  from Lemma \ref{2b}.

\begin{lemma}\label{2 lemma 3} Let $1/2<\alpha, T>0,$ and $p,q$ satisfy
$$
p>\frac{n}{2\alpha-1},\ \  2\alpha-1=\frac{2\alpha}{q}+\frac{n}{p}.
$$
 Assume that $f\in L^{r}(\mathbb{R}^{n})$ with $\frac{n}{2\alpha-1}<r\leq p.$ Then we have
$$
 \|e^{-t(-\triangle)^{\alpha}}f\|_{L^{q}([0,T]; L^{p})}\lesssim T^{1-\frac{n}{2\alpha}\left(\frac{1}{n}+\frac{1}{r}\right)}\|f\|_{L^{r}}.
$$
\end{lemma}

 Lemma \ref{2c} gives us the homogeneous  Strichartz  estimates of  equation (\ref{eq1d}) except  endpoint cases.
 To obtain the endpoint estimates we need the abstract Strichartz estimates of Keel-Tao \cite{M.Keel T.Tao}.

\begin{lemma}\label{2d} \cite{M.Keel T.Tao} Let $H$ be a Hilbert space and $X$ be a Banach space. Suppose that
 $U(t): H\longrightarrow L^{2}(X)$  obeys the energy estimate:
$$
\|U(t)f\|_{L^{2}(X)}\lesssim\|f\|_{H}
$$
 and the untruncated decay estimate, that is for some $\sigma>0,$
$$
\|U(t)(U(s))^{*}f\|_{L^{\infty}}\lesssim |t-s|^{-\sigma}\|f\|_{L^{1}}, \ \forall s\neq t.
$$
 Then the estimates
$$
\|U(t)f\|_{L^{q}_{t}L^{p}_{x}}\lesssim\|f\|_{H},
$$
$$
\left\|\int(U(s))^{*}F(s)ds\right\|_{H}\lesssim \|F\|_{L^{q'}_{t}L^{p'}_{x}}
$$
$$
\left\|\int_{s<t}U(t)(U(s))^{*}F(s)ds\right\|_{L^{q}_{t}L^{p}_{x}}\lesssim\|F\|_{L^{q_{1}'}_{t}L^{p_{1}'}_{x}}
$$
 hold for all $\sigma-$admissible triplets  $(q,p,2)$ and  $(q_{1},p_{1},2)$ with $q,q_{1}\geq 2,$
 $(q,p,\sigma)$ and $(q_{1},p_{1},\sigma)$ are not $ (2,\infty, 1).$
\end{lemma}
\vspace{0.1in} {\section{ Proofs of Main Results}}

 {{\it{3.1. Proof of Proposition \ref{a2}}}}. We only need to prove (\ref{eq1a}) for $I=[0. \infty)$ since the
 proofs  for  other cases are similar. Assume that $(q,p,2)$ is a $\frac{n}{2\alpha}-$admissible triplet with
 $q\geq 2$  and  $\left(q, p,\frac{n}{2\alpha}\right)$ is not $(2,\infty, 1).$  It follows from Lemma \ref{2b}
 that we have the energy estimate
\begin{equation}\label{2222}\|e^{-t(-\triangle)^{\alpha}}f\|_{L^{2}_{x}}\lesssim\|f\|_{L^{2}}, \forall  t>0,
\end{equation}
 and untruncated decay estimate
\begin{equation}\label{2223}\|e^{-(t+s)(-\triangle)^{\alpha}}f\|_{L^{\infty}}\lesssim|t+s|^{-\frac{n}{2\alpha}}\|f\|_{L^{1}}
  \lesssim|t-s|^{-\frac{n}{2\alpha}}\|f\|_{L^{1}},\ \forall  s\neq
  t, s, t\in(0, \infty).
\end{equation}
 By (\ref{2222}),  (\ref{2223}) and Lemma \ref{2a}, we can  apply   Lemma \ref{2d} with  $U(t)=e^{-t(-\triangle)^{\alpha}}$
 for $t>0,$ $H=L^{2}(\mathbb{R}^{n})$ and $X=\mathbb{R}^{n}$ to obtain  (\ref{eq1a}).

 \vspace{0.1in}
 {{\it{3.2. Proof of Theorem \ref{a5}}}}.  We only need to prove (\ref{eq1b})  for $I=[0. \infty)$, the proofs  for
 other cases being similar. Assume that  $(q,p,2)$ and $ (q_{1},p_{1},2)$ satisfy  $1\leq p_{1}'< p\leq \infty,$
 $1<q_{1}'<q<\infty$ and  $\frac{1}{q_{1}'}+\frac{n}{2\alpha}\left(\frac{1}{p_{1}'}-\frac{1}{p}\right)=1+\frac{1}{q}.$
  It follows from Lemma \ref{2b} that
$$
\|e^{-(t-s)(-\triangle)^{\alpha}}F(s,x)\|_{L_{x}^{p}}\lesssim|t-s|^{-\frac{n}{2\alpha}\left(\frac{1}{p_{1}'}-\frac{1}{p}\right)}
\|F(s, x)\|_{L_{x}^{p_{1}'}}, \ \ \forall s<t.
$$
 Then  the  Hardy-Littlewood-Sobolev inequality  implies that
\begin{eqnarray*}
 \left\|\int_{0}^{t}e^{-(t-s)(-\triangle)^{\alpha}}F(s, x)ds\right\|_{L^{q}_{t}(I;
 L^{p}_{x})}\!\!\!\!\!&\lesssim&\!\!\!\!
 \left\|\int_{0}^{t}\!\!\|e^{-(t-s)(-\triangle)^{\alpha}}F(s, x)\|_{L^{p}_{x}}ds\right\|_{L^{q}_{t}(I)}\nonumber\\
 &\lesssim&\!\!\!\!\left\|\int_{0}^{t}
 \!\!|t-s|^{-\frac{n}{2\alpha}\left(\frac{1}{p_{1}'}-\frac{1}{p}\right)}\!\|F(s,x)\|_{L^{p_{1}'}_{x}}ds\right\|_{L^{q}_{t}(I)}\nonumber\\
 &\lesssim&\!\!\!\!\|F\|_{L^{q_{1}'}_{t}(I; L^{p_{1}'}_{x})}.\nonumber
\end{eqnarray*}
This finishes the proof of (\ref{eq1b}).

 \vspace{0.1in}
 {{\it{3.3. Proof of Theorem \ref{a4}}}. Let $n=2\alpha.$ Define $\varphi\in C^{\infty}(\mathbb{R})$  with
 $\hbox{supp}(\varphi)\subseteq(1/2, 2),$ $\varphi(x)=1$ for $x\in (3/4, 9/8)$ and $\sum_{k\in \mathbb{Z}}\varphi(2^{-k}t)=1$
 for all $t>0.$ Let $\varphi_{k}(t)=\varphi(2^{-k}t).$ Define $P_{k}f=\mathcal {F}^{-1}(\mathcal {F}f(\cdot)\varphi_{k}(|\cdot|))$
 be a Littlewood-Paley decomposition with respect to $\varphi_{k}$ (see \cite{E. M.  Stein}).  Since $BMO=\dot{F}_{\infty}^{0,2}$
 (see  Frazier-Jawerth-Weiss \cite{F. Frazier Jawerth Weiss}),
 $$
 \|g\|_{BMO}\approx\left\|\left(\sum\limits_{k\in \mathbb{Z}}|P_{k}g|^{2}\right)^{1/2}\right\|_{L^{\infty}}.
 $$
 Let $M_{k}=B(0,2^{k+1})\backslash B(0,2^{k-1})$ and $\chi_{M_{k}}$ its characteristic function.
 Since $\varphi$ is supported in $(1/2, 2)$ and $n=2\alpha,$  we have
\begin{eqnarray*}
 \|e^{-t(-\triangle)^{\alpha}}P_{k}f\|^{2}_{L^{2}_{t}L^{\infty}_{x}}\nonumber
 &\leq&\!\!
 \int_{0}^{\infty}\sup\limits_{x}\left|\int_{\mathbb{R}^{n}}e^{-t|\xi|^{n}}
 e^{i\langle\xi,x\rangle}\hat{f}(\xi)\varphi(2^{-k}|\xi|)d\xi \right|^{2}dt \nonumber\\
 &\leq&\!\!
 \int_{0}^{\infty}\int_{\mathbb{R}^{n}}\chi_{M_{k}}(\xi)d\xi
 \sup\limits_{x}\int_{\mathbb{R}^{n}}\left|e^{-t|\xi|^{n}}
 e^{i\langle\xi,x\rangle}\hat{f}(\xi)\varphi(2^{-k}|\xi|)\right|^{2}d\xi dt \nonumber\\
 &\lesssim&\!\!
 2^{(k-1)n}(2^{2n}-1)\int_{0}^{\infty}\int_{M_{k}}e^{-2t|\xi|^{n}}
 \left|\hat{f}(\xi)\varphi(2^{-k}|\xi|)\right|^{2}d\xi dt \nonumber\\
 &\lesssim&\!\!2^{(k-1)n}(2^{2n}-1)\int_{0}^{\infty}e^{-t2^{(k-1)n+1}}dt\|f\|^{2}_{L^{2}}\nonumber\\
 &\lesssim&\!\!(2^{2n-1}-1/2)\|f\|^{2}_{L^{2}}\nonumber\\
 &\lesssim&\!\!\|f\|^{2}_{L^{2}}.\nonumber
\end{eqnarray*}
 Take $\psi\in C^{\infty}(\mathbb{R})$  with $\hbox{supp}(\psi)\subseteq(1/4,4)$ and $\psi(x)\varphi(x)=\varphi(x).$  Define
 $$
 \widetilde{P}_{k}f=\mathcal {F}^{-1}((\mathcal {F}f)\psi_{k}).
 $$
  Then we have
 \begin{eqnarray*}
  \|e^{-t(-\triangle)^{\alpha}}f\|^{2}_{L^{2}_{t}((0,\infty); BMO_{x})}\nonumber
  &\lesssim&\int_{0}^{\infty}\sup\limits_{x}\left(\sum\limits_{k}|e^{-t(-\triangle)^{\alpha}}P_{k}f|^{2}\right)dt\\
 \nonumber
  &\lesssim& \sum\limits_{k}\|e^{-t(-\triangle)^{\alpha}}P_{k}\widetilde{P}_{k}f\|^{2}_{L^{2}_{t}L_{x}^{\infty}}\nonumber\\
  &\lesssim&
  \sum\limits_{k}\|\widetilde{P}_{k}f\|^{2}_{L^{2}}\nonumber\\
  &\lesssim&\|f\|_{L^{2}}^{2}.\nonumber
  \end{eqnarray*}
 That is,   (\ref{eq1i}) holds.

\vspace{0.1in}
 {{\it{3.4. Proof of Theorem \ref{111111}}}. (a). Let $1\leq r\leq p\leq \infty$ and $n<2\alpha.$ It follows from Lemma \ref{2b} that
$$
 s^{-\frac{nr}{2p\alpha}}\|e^{-s(-\triangle)^{\alpha}}f\|^{r}_{_{L^{p}(\mathbb{R}^{n})}}
 \lesssim s^{-\frac{n}{2\alpha}}\|f\|^{r}_{L^{r}(\mathbb{R}^{n})}.
$$
 On the other hand,  $n<2\alpha$ implies that
$$
 \int_{0}^{T}s^{-\frac{n}{2\alpha}}ds=\frac{2\alpha}{2\alpha-n}T^{1-\frac{n}{2\alpha}}.
$$
 Thus (\ref{222222}) holds.\\
 (b). The following  proof is essentially  the same as  the proof Tao's \cite[Lemma 2.5]{T. Tao 1}. For the sake of completeness, it  is
 provided here. We use the $TT^{*}$ method. Thus, by duality and the  self-adjointness of  $e^{-t(-\triangle)^{\alpha}}$ it suffices to verify
\begin{equation}\label{eq333}
\left\|\int_{0}^{\infty}s^{-1/p}e^{-s(-\triangle)^{\alpha}}F(s,x)ds\right\|^{2}_{L^{2}(\mathbb{R}^{n})}
\lesssim \int_{0}^{\infty}\|F(s,x)\|^{2}_{L^{p'}_{x}(\mathbb{R}^{n})}ds
\end{equation}
 for all test functions $F.$ The left hand side of (\ref{eq333}) can be written as
$$
 \int_{0}^{\infty}\int_{0}^{\infty}s_{1}^{-1/p}s^{-1/p}\left\langle e^{-\frac{s+s_{1}}{2}(-\triangle)^{\alpha}}F(s,x),
 e^{-\frac{s+s_{1}}{2}(-\triangle)^{\alpha}}F(s_{1},x)\right\rangle_{x}dsds_{1}.
$$
 Let $g(s)=\|F(s,x)\|_{L^{p'}_{x}(\mathbb{R}^{n})}.$ According to Lemma \ref{2b}, we have
$$
 \left|\left\langle e^{-\frac{s+s_{1}}{2}(-\triangle)^{\alpha}}F(s,x),
 e^{-\frac{s+s_{1}}{2}(-\triangle)^{\alpha}}F(s_{1},x)\right\rangle_{x}\right|
 \lesssim(s+s_{1})^{-2\left(\frac{1}{p'}-\frac{1}{2}\right)}g(s)g(s_{1}).
$$
  Hence, it suffices to prove that
\begin{equation}\label{eq334}
 \int_{0}^{\infty}\int_{0}^{\infty}\frac{g(s)g(s_{1})dsds_{1}}{(s+s_{1})^{1-2/p}s^{1/p}s_{1}^{1/p}}
 \lesssim \int_{0}^{\infty}g(s)^{2}ds.
\end{equation}
 On the other hand,   by symmetry we can only consider the region $s_{1}\leq s$ which can be decomposed into the dyadic
 ranges $2^{-m}s\leq s_{1}\leq 2^{-m+1}s.$ Hence the left hand side of (\ref{eq334}) can be bounded by
\begin{eqnarray*}
 &\lesssim&\sum\limits^{\infty}_{m=1}2^{m/p}\int_{0}^{\infty}\int_{2^{-m}s\leq s_{1}\leq 2^{-m+1}s}\frac{g(s)g(s_{1})}{s}ds_{1}ds\\
 &\lesssim&\sum\limits^{\infty}_{m=1}2^{m\left(\frac{1}{p}-\frac{1}{2}\right)}\int_{0}^{\infty}g(s)^{2}ds\\
 &\lesssim&\int_{0}^{\infty}g(s)^{2}ds
\end{eqnarray*}
 with  the second inequality using the Schur's test of Tao \cite{T. Tao 2}.

\vspace{0.1in}
  {{\it{3.5. Proof of Theorem \ref{a8}}}}.
  We only need to prove (\ref{eq1s}) for $Z=\dot{H}^{\alpha, p}_{x}.$ Suppose
 $$
 n>2\alpha>0,  p\in[1, 2), q \in (1, 2),\frac{1}{q}+\frac{n}{2\alpha}\left(\frac{1}{p}-\frac{1}{2}\right)=\frac{3}{2}.
 $$
 Thus $\frac{n}{2\alpha}\left(\frac{1}{p}-\frac{1}{2}\right)\in (0,1).$ According to the imbedding of  $\dot{H}^{\alpha, 2}$ into
 $L^{\frac{2n}{n-2\alpha}},$  Lemmas \ref{2a} \& \ref{2b}  and the Hardy-Littlewood-Sobolev inequality, we obtain
\begin{eqnarray*}
 \left\|\!\int_{0}^{t}\!\!e^{-(t-s)(-\triangle)^{\alpha}}\!F(s,x)ds\right\|_{L^{2}_{t}(L^{\frac{2n}{n-2\alpha}}_{x})}
 \!\!\!\!&\lesssim&\!\!\!\!
 \left\|\!\int_{0}^{t}\!\!e^{-(t-s)(-\triangle)^{\alpha}}F(s,x)ds\right\|_{L^{2}_{t}(\dot{H}^{\alpha,
 2}_{x})}
 \nonumber\\
 &\lesssim&\!\!\!\!\left\|(-\triangle)^{\alpha/2}\!\!\int_{0}^{t}\!\!e^{-(t-s)(-\triangle)^{\alpha}}F(s,x)ds\right\|_{L^{2}_{t}(L^{2}_{x})}
 \nonumber\\
 &\lesssim&\!\!\!\!\left\|\int_{0}^{t}\!\!e^{-(t-s)(-\triangle)^{\alpha}}((-\triangle)^{\alpha/2}F(s,x))ds\right\|_{L^{2}_{t}(L^{2}_{x})}
 \nonumber\\
 &\lesssim&\!\!\!\!\left\|\int_{0}^{t}\!\!\|e^{-(t-s)(-\triangle)^{\alpha}}((-\triangle)^{\alpha/2}F(s,x))\|_{L^{2}_{x}}ds\right\|_{L^{2}_{t}}
 \nonumber\\
 &\lesssim&\!\!\!\!\left\|\int_{0}^{t}\!\!|t-s|^{-\frac{n}{2\alpha}\left(\frac{1}{p}-\frac{1}{2}\right)}\|(-\triangle)^{\alpha/2}F(s,x)\|_{L^{p}_{x}}ds\right\|_{L^{2}_{t}}
 \nonumber\\
 &\lesssim&\!\!\!\!\|(-\triangle)^{\alpha/2}F\|_{L^{q}_{t}(L^{p}_{x})}
 \nonumber\\
 &\lesssim&\!\!\!\! \|F\|_{L^{q}_{t}(\dot{H}^{\alpha,
 p}_{x})}.\nonumber
\end{eqnarray*}
 This finishes the proof of (\ref{eq1s}).

 \vspace{0.1in}
 {{\it{3.6. Proof of Corollary \ref{a7}}}. We only check  (\ref{eq1k})  with $(X_{1},X_{2})=(\dot{B}^{s}_{p, 2}, \dot{B}^{s}_{2, 2})$
  and (\ref{eq1m}) with $(Y_{1}, Y_{2})=(\dot{B}^{s}_{p,2}, \dot{B}^{s}_{ p_{1}', 2})$  because the proofs of other cases are
  similar. We assume that $p<\infty$  since  the case $p=\infty$ is  similar. We will use the following equivalent norms in Besov spaces.
  Let $\eta$ be an infinitely differential function with compact support in $\mathbb{R}^{n}$ satisfying
\begin{equation}\label{eq1e}
 \eta(\xi)=\left \{\begin{array}{ll}
 1, & |\xi|\leq 1, \\
 0, & |\xi|\geq 2, \end{array} \right.
\end{equation}
 define the sequence $\{\psi_{j}\}_{j\in\mathbb{Z}}$ in  $\mathscr {S}(\mathbb{R}^{n})$ by
\begin{equation}\label{eq1f}
\psi_{j}(\xi)=\eta\left(\frac{\xi}{2^{j}}\right)-\eta\left(\frac{\xi}{2^{j-1}}\right).
\end{equation}
 Through this sequence,   the norms in the inhomogeneous and homogeneous Besov spaces $B^{s}_{p,q}(\mathbb{R}^{n})$ and
  $\dot{B}^{s}_{p,q}(\mathbb{R}^{n})$  for $1\leq p, q\leq \infty$ and $s\in \mathbb{R}^{n}$ are equivalent to
 \begin{equation}\label{eq1g}
 \|f\|_{B^{s}_{p,q}}=\|\mathcal {F}^{-1}(\eta\mathcal
 {F}(f))\|_{L^{p}(\mathbb{R}^{n})}+ \left \{\begin{array}{ll}
 \left(\sum\limits_{j=1}^{\infty}(2^{sj}\|\mathcal
 {F}^{-1}(\psi_{j}\mathcal
 {F}(f))\|_{L^{p}(\mathbb{R}^{n})})^{q}\right)^{1/q}
 & \hbox{if} \  q< \infty, \\
 \sup\limits_{j\geq 1}2^{sj}\|\mathcal {F}^{-1}(\psi_{j}\mathcal
 {F}(f))\|_{L^{p}(\mathbb{R}^{n})} & \hbox{if}\ q=\infty \end{array}
\right.
\end{equation}
 and
 \begin{equation}\label{eq1h}
 \|f\|_{\dot{B}^{s}_{p,q}}= \left \{\begin{array}{ll}
 \left(\sum\limits_{j=-\infty}^{\infty}(2^{sj}\|\mathcal
 {F}^{-1}(\psi_{j}\mathcal
 {F}(f))\|_{L^{p}(\mathbb{R}^{n})})^{q}\right)^{1/q}
 & \hbox{if}\  q< \infty, \\
 \sup\limits_{j\in \mathbb{Z}}2^{sj}\|\mathcal
 {F}^{-1}(\psi_{j}\mathcal {F}(f))\|_{L^{p}(\mathbb{R}^{n})} &
 \hbox{if}\ q=\infty.
 \end{array}
  \right.
\end{equation}

 {{\it{Part 1. Proof of (\ref{eq1k})}}}.  We assume that $q<\infty,$ note that the case $q=\infty$ is obvious. Define
 $u(t)=e^{-t(-\triangle)^{\alpha}}f.$ Then
 $$
 \mathcal {F}^{-1}(\psi_{j}\mathcal {F}(u))=\mathcal {F}^{-1}(e^{-t|\xi|^{2\alpha}}\psi_{j}\mathcal {F}(f))
 =e^{-t(-\triangle)^{\alpha}}(\mathcal {F}^{-1}(\psi_{j}\mathcal {F}(f))).
 $$
  Hence
 $$
  \|u\|_{L^{q}_{t}(I;\dot{B}^{s}_{p,2})}^{2}=\left(\int\limits_{I}\left(\sum\limits_{j}2^{2sj}
  \|e^{-t(-\triangle)^{\alpha}}(\mathcal {F}^{-1}(\psi_{j}\mathcal {F}(f)))\|_{L^{p}}^{2}\right)^{q/2}dt\right)^{2/q}.
 $$
  Letting
  $
  A_{j}(t)=2^{2sj}\|e^{-t(-\triangle)^{\alpha}}(\mathcal {F}^{-1}(\psi_{j}\mathcal {F}(f)))\|_{L^{p}}^{2}
  $
  and $k=q/2\geq 1,$ we have
 \begin{eqnarray*}
  \|u\|_{L^{q}_{t}(I; \dot{B}^{s}_{p,2})}^{2}&=&\left(\int_{I}\left(\sum\limits_{j}A_{j}(t)\right)^{k}dt\right)^{1/k}\nonumber\\
  &=&\|\sum\limits_{j}A_{j}(\cdot)\|_{L^{k}(I)}\nonumber\\
  &\leq& \sum\limits_{j}\|A_{j}(\cdot)\|_{L^{k}(I)}\nonumber\\
  &=& \sum_{j}2^{2sj}
  \|e^{-t(-\triangle)^{\alpha}}(\mathcal {F}^{-1}(\psi_{j}\mathcal {F}(f)))\|^{2}_{L^{q}(I;L^{p})}.\nonumber
  \end{eqnarray*}
  Using Proposition \ref{a2}, we deduce
$$
 \|u\|_{L^{q}_{t}(I;\dot{B}^{s}_{p,2} )}\lesssim \left(\sum\limits_{j}2^{2sj}
 \|\mathcal {F}^{-1}(\psi_{j}\mathcal {F}(f))\|^{2}_{L^{2}}\right)^{1/2}\lesssim \|f\|_{\dot{B}^{s}_{2,2}}.
$$
 Therefore, (\ref{eq1k}) holds.\\
 {{{\it{Part 2. Proof of (\ref{eq1m})}}}}. Let $u(t)=\int_{0}^{t}e^{-(t-s)(-\triangle)^{\alpha}}F(s,x)ds.$ Then
\begin{eqnarray*}
  2^{sj}\mathcal {F}^{-1}(\psi_{j}\mathcal {F}(u))&=&2^{sj}\mathcal {F}^{-1}\int_{0}^{t}
  \psi_{j}\mathcal {F}(e^{-(t-s)(-\triangle)^{\alpha}}F(s,x))ds\nonumber\\
  &=&2^{sj}\mathcal {F}^{-1}\int_{0}^{t}
  e^{-(t-s)|\xi|^{2\alpha}}\psi_{j}\mathcal {F}(F(s,\xi))ds\nonumber\\
  &=&2^{sj}\int_{0}^{t}
  \mathcal {F}^{-1}\left(e^{-(t-s)|\xi|^{2\alpha}}\psi_{j}\mathcal {F}(F(s,\xi))\right)ds\nonumber\\
  &=&\int_{0}^{t}
  e^{-(t-s)(-\triangle)^{\alpha}}\left(2^{sj}\mathcal {F}^{-1}\left(\psi_{j}\mathcal {F}(F(s,\xi))\right)\right)ds\nonumber\\
  &=&\int_{0}^{t}
  e^{-(t-s)(-\triangle)^{\alpha}}v_{j}(t)ds,\nonumber
\end{eqnarray*}
where
 $v_{j}(t)=2^{sj}\mathcal {F}^{-1}\left(\psi_{j}\mathcal {F}(F(s,\xi))\right).$
 Thus
$$
\|u\|_{L^{q}_{t}(I;\dot{B}^{s}_{p,2})}^{2}\lesssim\left(\int_{I}\left(\sum\limits_{j}\left\|\int_{0}^{t}
e^{-(t-s)(-\triangle)^{\alpha}}v_{j}(t)ds\right\|_{L^{p}}^{2}\right)^{q/2}dt\right)^{2/q}.
$$
 In a similar manner to verify  (\ref{eq1k}), we have
$$
 \|u\|_{L^{q}_{t}(I;\dot{B}^{s}_{p,2})}^{2}\lesssim\sum\limits_{j}\left\|\int_{0}^{t}
 e^{-(t-s)(-\triangle)^{\alpha}}v_{j}(t)ds\right\|_{L^q_{t}(I; L^{p})}^{2}.
$$
  Applying  Theorem \ref{a5}, we get
$$
 \|u\|_{L^{q}_{t}(I;\dot{B}^{s}_{p,2})}^{2}\lesssim\sum\limits_{j}
 \|v_{j}\|^{2}_{L^{q_{1}'}_{t}(I; L^{p_{1}'})}\lesssim\sum\limits_{j}\left(\int_{I}R_{j}(t)dt\right)^{k},
$$
  where $R_{j}(t)=\|v_{j}(t)\|^{q_{1}'}_{L^{p_{1}'}}$ and $k=2/q_{1}'\geq 1.$ An application of the Minkowski inequality yields
\begin{eqnarray*}
 \|u\|_{L^{q}_{t}(I;\dot{B}^{s}_{p,2})}^{2/k}&\lesssim&\left\|\int_{I}R_{j}(t)dt\right\|_{{l}^{k}(\mathbb{Z})}\nonumber\\
 &\lesssim&\int_{I}\left\|R_{j}(t)\right\|_{{l}^{k}(\mathbb{Z})}dt\nonumber\\
 &\lesssim&\int_{I}\left(\sum\limits_{j}\|v_{j}(t)\|^{2}_{L^{p_{1}'}}\right)^{q_{1}'/2}dt\nonumber\\
 &\lesssim&\|F\|_{L^{q_{1}'}_{t}(I;\dot{B}^{s}_{p_{1}',2})}^{q_{1}'}.\nonumber
\end{eqnarray*}
 Thus (\ref{eq1m}) holds.

 \vspace{0.1in}
 {{\it{3.7. Proof of Corollary \ref{88}}}}. We shall prove  this theorem for $n>2\alpha.$  In the case $n=2\alpha,$
 we can replace in the sequel the space $L^{2}_{t}(J; L^{\frac{2n}{n-2\alpha}}_{x})$ by any $L^{q}_{t}(J;L^{p}_{x})$
 for $1-$admissible $(q,p,2)$  with  $p$ arbitrarily large.\\
 We consider the following two cases.

 {{\em{Case 1, $r\in(2, \infty)$}}}: Let $(q,p,2)$  ($2\leq q<\infty$) be  $\frac{n}{2\alpha}-$admissible.  Let
 $J=[0, \varepsilon]$ where $\varepsilon>0$  will be determined later and $(k,l,2)$ be  $\frac{n}{2\alpha}-$admissible
 with $q\leq k<\infty$,  and set
$$
 X=L^{k}_{t}(J;L^{l}_{x})\cap L^{2}_{t}(J;L^{\frac{2n}{n-2\alpha}}_{x})\ \hbox{with}\ \|v\|_{X}:=\max\left\{\|v\|_{L^{k}_{t}(J;L^{l}_{x})},
\|v\|_{L^{2}_{t}(J;L^{\frac{2n}{n-2\alpha}}_{x})}\right\}.
$$
 By interpolation (see Triebel \cite{H. Triebel}), $X$  can be  embedded into $L^{q_{0}}_{t}(J;L^{p_{0}}_{x})$ for  each
 $\frac{n}{2\alpha}-$admissible triplet $(q_{0},p_{0},2)$ with  $2\leq q_{0}\leq k.$  Define $T(v)$ on $X$ by
$$
T(v)=e^{-t(-\triangle)^{\alpha}}f+\int_{0}^{t}e^{-(t-s)(-\triangle)^{\alpha}}(F(s,x)-V(s,x)v(s,x))ds, \forall v=v(t,x)\in X.
$$
 Applying Proposition \ref{a2} and Theorem \ref{a5}, we have
$$
\|T(v)\|_{L^{{q_{0}}}_{t}(J; L^{{p_{0}}}_{x})}\leq C
\|f\|_{2}+ C\|F\|_{L^{{q_{1}'}}_{t}(J;L^{p_{1}'}_{x})}
+C\|Vv\|_{L^{q_{2}'}_{t}(J;L^{p_{2}'}_{x})},
$$
 for all $\frac{n}{2\alpha}-$admissible triplets $(q_{0},p_{0},2), (q_{1},p_{1},2),$ and $ (q_{2},p_{2},2)$ satisfying
 $$
 2\leq q_{0}\leq k,\ \ q_{1}'\in (1, 2),\ \ q_{2}'\in (1, 2),\
\  1\leq p_{1}'<p_{0}\leq \infty,\ \ 1\leq p_{2}'<p_{0}\leq \infty.
  $$
 Here and later $C>0$ is a constant.  Clearly,  H$\ddot{o}$lder's inequality implies
$$
\|T(v)\|_{L^{{q_{0}}}_{t}(J; L^{{p_{0}}}_{x})}\leq C \|f\|_{2}+
C\|F\|_{L^{{q_{1}'}}_{t}(J;L^{p_{1}'}_{x})} +C\|V\|_{L^{r}_{t}(J;
L^{s}_{x})}\|v\|_{L^{2}_{t}(J;L^{\frac{2n}{n-2\alpha}}_{x})}
$$
provided
$$
 \frac{1}{q_{2}}=\frac{1}{2}-\frac{1}{r}, \frac{1}{p_{2}}=\frac{n+2\alpha}{2n}-\frac{1}{s}.
$$
 This and the  assumption on $r$ and $s$ imply that $q_{2}'\in (1,2) ,$ $p_{2}'\in[1,2)$ and
$$
\frac{1}{q_{2}}+\frac{n}{2\alpha}\frac{1}{p_{2}}=\frac{1}{2}
+\frac{n}{2\alpha}\frac{n+2\alpha}{2n}-\left(\frac{n}{2\alpha}\frac{1}{s}+\frac{1}{r}\right)=
\frac{n}{4\alpha}.
$$
 Taking $(q_{0},p_{0}, 2)$ be  $(k, l,2)$ and $(2,\frac{2n}{n-2\alpha}, 2),$ we get
$$
 \|T(v)\|_{X}\leq C
 \|f\|_{2}+ C\|F\|_{L^{{q_{1}'}}_{t}(J;L^{p_{1}'}_{x})}
 +C\|V\|_{L^{r}_{t}(J; L^{s}_{x})}\|v\|_{X}.
$$
 Hence   $T(v)\in X$ and $T$ is a operator from $X$ to $X.$ Since $r<\infty,$ we may  choose  such an  $\varepsilon>0$  that
 \begin{equation}\label{3a}
  C\|V\|_{L^{r}_{t}(J; L^{s}_{x})}\leq\frac{1}{2}.
 \end{equation}
  This fact yields that
$$
 \|T(v_{1})-T(v_{2})\|_{X}\leq \frac{1}{2}\|v_{1}-v_{2}\|_{X}, \ \ \forall v_{1}, v_{2}\in X.
$$
 Thus $T$ is a contraction operator on $X,$ and  $T$ has a unique fixed point $v(t,x)$ which is the unique solution
 of equation (\ref{eq1111d}) and $v$ satisfies
 $$
 \|v\|_{X}\lesssim \|f\|_{2}+ \|F\|_{L^{{q_{1}'}}_{t}(J;L^{p_{1}'}_{x})}.
$$
 Since $X$ is embedded  in $L^{q}_{t}(J;L^{p}_{x}),$ one finds
 \begin{equation*}\label{3b}
 \|v\|_{L^{{q}}_{t}(J; L^{{p}}_{x})}\lesssim
 \|f\|_{2}+ \|F\|_{L^{{q_{1}'}}_{t}(J;L^{p_{1}'}_{x})}.
 \end{equation*}

 Now, we can apply the previous arguments to any subinterval $J=[t_{1},t_{2}]$ on which   a condition like (\ref{3a}) holds and
 obtain
 \begin{equation}\label{3b}
  \|v\|_{L^{q}_{t}(J;L^{p}_{x})}\lesssim\|v(t_{1})\|_{L^{2}}+\|F\|_{L^{{q_{1}'}}_{t}(J;L^{p_{1}'}_{x})}.
 \end{equation}
 Note that (\ref{3b}) implies
\begin{equation}\label{3c}
 \|v(t_{2})\|_{L^{2}}\lesssim\|v(t_{1})\|_{L^{2}}+\|F\|_{L^{{q_{1}'}}_{t}(J;L^{p_{1}'}_{x})}.
\end{equation}
 If $I=[0,T]$ for $0<T<\infty,$  we can partition $I$ into  a finite many of subintervals on which the condition (\ref{3a}) holds.
 If $I=[0,\infty),$ since $V\in L^{r}_{t}(I;L^{s}_{x}(\mathbb{R}^{n}))$ we can find $T_{1}>0$ such that
 $C\|V\|_{L^{r}_{t}((T_{1},\infty);L^{s}_{x}(\mathbb{R}^{n}))}<\frac{1}{2}$ and partition $[0,T_{1}]$ similarly.
 Thus we can prove (\ref{1111c}) by inductively  applying  (\ref{3b}) and  (\ref{3c}).

 {{\em{Case 2, $r\in (1,2)$}}}.
  Since $(r, \frac{2s}{s+2})$ is the dual of $(r', \frac{2s}{s-2}),$ our assumption on $r,s$ implies
$$
 \frac{1}{r'}+\frac{n}{2\alpha}\frac{s-2}{2s}=\frac{n}{2\alpha s}+\frac{n}{2\alpha}\frac{s-2}{2s}=\frac{n}{4\alpha}.
$$
 Thus $(r', \frac{2s}{s-2})$ is  $\frac{n}{2\alpha}-$admissible with $r\in (1,2).$ In a fashion analogous  to handling  {\em{Case 1}},
 we use  Theorems \ref{a2} \& \ref{a5}, to obtain
 $$
 \|T(v)\|_{L^{{q_{0}}}_{x}(J; L^{{p_{0}}}_{x})}\leq C
 \|f\|_{2}+ C\|F\|_{L^{{q_{1}'}}_{t}(J;L^{p_{1}'}_{x})}
 +C\|Vv\|_{L^{r}_{t}(J;L^{\frac{2s}{s+2}}_{x})}.
 $$
  Again, by  H$\ddot{o}$lder's inequality we have
 $$
  \|T(v)\|_{L^{{q_{0}}}_{t}(J; L^{{p_{0}}}_{x})}\leq C
  \|f\|_{2}+ C\|F\|_{L^{{q_{1}'}}_{t}(J;L^{p_{1}'}_{x})}
  +C\|V\|_{L^{r}_{t}(J;L^{s}_{x})}\|v\|_{L^{\infty}_{t}(J;L^{2}_{x})}.
 $$
 Similarly, taking $(q_{0},p_{0}, 2)$ be $(k, l,2)$ and $(2,\frac{2n}{n-2\alpha}, 2)$, we have
$$
 \|T(v)\|_{X}\leq C
 \|f\|_{2}+ C\|F\|_{L^{{q_{1}'}}_{t}(J;L^{p_{1}'}_{x})}
 +C\|V\|_{L^{r}_{t}(J; L^{s}_{x})}\|v\|_{X}.
$$
 The rest of the  proof is similar to that of the  first case.\\

 \vspace{0.1in} {{\it{3.8. Proof of Theorem \ref{a9}}}}.  We only prove the case $(X,Y)=(L^{p}_{x},L^{p'_{1}}_{x})$ since similar
 arguments apply to other cases. Assume that $T\in (0, \infty),$ $1\leq p_{1}'<p\leq\infty, $ $1\leq q_{1}'<q\leq\infty, $
 $\frac{1}{r}=\frac{1}{p}+\frac{1}{p_{1}},$ $\frac{1}{h}=\frac{1}{q}+\frac{1}{q_{1}}$ and $\frac{nh}{2\alpha}\left(1-\frac{1}{r}\right)\in (0, 1).$
  Let $I=[0,T).$ According to   the Young's inequality and the definition of $e^{-t(-\triangle)^{\alpha}}$, we have
\begin{eqnarray*}
\left\|\int_{0}^{t}e^{-(t-s)(-\triangle)^{\alpha}}F(s,x)ds\right\|_{L^{q}_{t}(I;
L^{p}_{x})}\nonumber
\!\!&\lesssim&\!\left\|\int_{0}^{t}\!\|e^{-(t-s)(-\triangle)^{\alpha}}F(s,x)\|_{L^{p}_{x}}ds\right\|_{L^{q}_{t}(I)}\\\nonumber
&\lesssim&\left\|\int_{0}^{t}\!\|K_{t-s}^{\alpha}(x)\ast_{x}
F(s,x)\|_{L^{p}_{x}}ds\right\|_{L^{q}_{t}(I)}\\\nonumber
&\lesssim&\left\|\int_{0}^{t}\!\|K_{t-s}^{\alpha}(x)\|_{L^{r}_{x}}
\|F(s,x)\|_{L^{p_{1}'}_{x}}ds\right\|_{L^{q}_{t}(I)}\\
\nonumber &\lesssim&\|K_{t}^{\alpha}(x)\|_{L^{h}_{t}(I; L^{r}_{x})}
\|F(s,x)\|_{L^{q_{1}'}_{t}(I;L^{p_{1}'}_{x})}.\nonumber
\end{eqnarray*}
 Thus it suffices to prove $ \|K_{t}^{\alpha}(x)\|_{L^{h}_{t}(I; L^{r}_{x})}\lesssim T^{\frac{1}{h}-\frac{n}{2\alpha}(1-\frac{1}{r})}. $
 In fact, it follows from Miao-Yuan-Zhang's \cite[Lemma 2.1]{C. Miao} that $K_{1}^{\alpha}(x)\in L^{k}$ for all $1\leq k \leq \infty.$
 Since $\frac{1}{r}=\frac{1}{p}+\frac{1}{p_{1}}$ and $p_{1}'<p$ imply that  $r>1,$  $K_{1}^{\alpha}(x)\in L^{r}.$  Hence
  \begin{eqnarray*}
 \|K_{t}^{\alpha}(x)\|_{L^{h}_{t}(I; L^{r}_{x})}\nonumber
 &=&\left(\int_{0}^{T}\left(\int_{\mathbb{R}^{n}}\left(
 \int_{\mathbb{R}^{n}}e^{ix\cdot
 \xi}e^{-t|\xi|^{2\alpha}}d\xi\right)^{r}
 dx\right)^{\frac{h}{r}}dt\right)^{1/h}\\\nonumber
 &=&\left(\int_{0}^{T}t^{-\frac{nh}{2\alpha}(1-\frac{1}{r})}dt\right)^{1/h}\|K_{1}^{\alpha}\|_{L^{r}}\nonumber\\
 &\lesssim& T^{\frac{1}{h}-\frac{n}{2\alpha}(1-\frac{1}{r})}.
  \end{eqnarray*}
 This finishes the proof of Theorem  \ref{a9}.

 \vspace{0.1in}
  {{\it{3.9 Proof of Proposition  \ref{propostion}.}}  By Lemma \ref{2b} and $L^{p}-$boundness  of Riesz transform, we have
\begin{eqnarray*}
 \|B(u,v)\|_{L^{p}}&\lesssim&\int_{0}^{t}\|\nabla
 e^{-(t-s)(-\triangle)^{\beta}}P(u(s,\cdot)\otimes v(s,\cdot))\|_{L^{p}}ds\\
 &\lesssim&\int_{0}^{t}\frac{1}{|t-s|^{\frac{1}{2\alpha}+\frac{n}{2\alpha}\left(\frac{2}{p}-\frac{1}{p}\right)}}
 \|(u(s,\cdot)\otimes v(s,\cdot))\|_{L^{p/2}}ds\\
 &\lesssim&\int_{0}^{t}\frac{1}{|t-s|^{\frac{1}{2\alpha}+\frac{n}{2p\alpha}}}
 \|u(s,\cdot)\|_{L^{p}}\|v(s,\cdot)\|_{L^{p}}ds.
\end{eqnarray*}
 Since $\alpha>\frac{1}{2}$ and $p>\frac{n}{2\alpha-1},$
$$
 0<\frac{1}{2\alpha}+\frac{n}{2p\alpha}<1.
$$
 It follows from  $2\alpha-1=\frac{2\alpha}{q}+\frac{n}{p}$ and the Hardy-Littlewood-Sobolev inequality that
\begin{eqnarray*}
 \|B(u,v)\|_{L^{q}([0,T];L^{p})}&\lesssim&
 \left\|(\|u(s,\cdot)\|_{L^{p}}\|v(s,\cdot)\|_{L^{p}})\right\|_{L^{q/2}([0,T])}\\
 &\lesssim&
 \|u\|_{L^{q}([0,T];L^{p})}\|v\|_{L^{q}([0,T];L^{p})}.\\
\end{eqnarray*}

 \vspace{0.1in}
 {{\it{3.10 Proof of Proposition  \ref{10}.}}  (a)  Under the assumption of (a), let
 $X=L^{q}([0,T]; L^{p}(\mathbb{R}^{n})).$   Define
  \begin{equation}\label{operator}
  Tv=e^{-t(-\triangle)^{\alpha}}g+ \int_{0}^{t}e^{-(t-s)(-\triangle)^{\alpha}}P(h-\nabla\cdot(v\otimes v)(s,x)ds.
\end{equation}
We will prove that if
 $$
  a:= T^{1-\frac{n}{2\alpha}\left(\frac{1}{n}+\frac{1}{r}\right)}\|g\|_{L^{r}(\mathbb{R}^{n})}+
 T^{\frac{1}{q}+\frac{1}{q_{1}}-\frac{n}{2\alpha}\left(\frac{1}{p_{1}'}-\frac{1}{p}\right)} \|h\|_{L^{q_{1}'}([0,T]; L^{p_{1}'}(\mathbb{R}^{n}))}
$$
 is bounded by an appropriate constant,  then $T$ is a contraction operator on the ball $B_{R}$ in $X$ with radius $R=2a.$
 For any  $v^{1},$ $v^{2}\in B_{R},$ we have

 \begin{eqnarray*}\label{371}
 \|T(v_{1})-T(v_{2})\|_{X}&=&\left\|\int_{0}^{t}\!e^{-(t-s)(-\triangle)^{\alpha}}P\nabla\cdot(v_{1}\otimes
  v_{1})ds-
  \!\int_{0}^{t}\!e^{-(t-s)(-\triangle)^{\alpha}}P\nabla\cdot(v_{2}\otimes
  v_{2})ds\right\|_{X}\\
  &=&\|B(v_{1}-v_{2},v_{1})-B(v_{2},v_{1}-v_{2})\|_{X}\\
  &\leq& \|B(v_{1}-v_{2},v_{1})\|_{X} +\|B(v_{2},v_{1}-v_{2})\|_{X},
 \end{eqnarray*}
 where
$$
 B(u,v)=\int_{0}^{t}(e^{-(t-s)(-\triangle)^{\alpha}})P\nabla\cdot(u\otimes v)(s)ds.
$$
 It follows from  Proposition \ref{propostion} that $B$ is bounded on $X.$ Thus
 $$
 \|T(v_{1})-T(v_{2})\|_{X}\leq
 C\|v_{1}-v_{2}\|_{X}\|v_{1}\|_{X}+C\|v_{2}\|_{X}\|v_{1}-v_{2}\|_{X},
 $$
 where $C>0$ is only  dependent on $\alpha, p$ and $q.$  Thus
 $$
 \|T(v_{1})-T(v_{2})\|_{X}\leq C(\|v_{1}\|_{X}+\|v_{2}\|_{X})\|v_{1}-v_{2}\|_{X}\leq CR\|v_{1}-v_{2}\|_{X}.
 $$
 To estimate $\|Tv\|_{X}$ for $v\in B_{R},$  we  use
 $$
 T(0)=e^{-t(-\triangle)^{\alpha}}g+\int_{0}^{t}e^{-(t-s)(-\triangle)^{\alpha}}Ph(s,x)ds
 $$
  to obtain $\|T(0)\|_{X}\leq Ca$  according to Theorem \ref{a9} and   Lemma \ref{2 lemma 3}. Consequently,
 $$
 \|T(v)\|_{X}=\|T(v)-T(0)+T(0)\|_{X}\leq \|T(v-0)\|_{X}+\|T(0)\|_{X}\leq CR\|v\|_{X}+Ca.
 $$
Since $a$ is  bounded by a suitable constant,  then we have
 $$
 \|T(v^{1})-T(v^{2})\|_{X}\leq\frac{1}{2}\|v^{1}-v^{2}\|_{X}\  \hbox{and}\ \|T(v)\|_{X}\leq R.
 $$
 It follows from the Banach contraction mapping principle that there exists a unique  $v\in X=L^{q}_{t}([0,T]; L^{p}_{x}(\mathbb{R}^{n})).$ \\
 (b) Note that $\frac{n}{p} +\frac{2\alpha}{q}=2\alpha-1$ implies that $(q,p,\frac{n}{2\alpha-1})$ is $\frac{n}{2\alpha}-$admissible.
 By Lemma \ref{2c}, we get
 $$
  \|e^{-t(-\triangle)^{\alpha}}g\|_{L^{q}_{t}([0,\infty); L^{p}_{x})}\lesssim \|g\|_{L^{\frac{n}{2\alpha-1}}}.
 $$
  On the other hand, Theorem \ref{a5} implies
$$
 \left\|\int_{0}^{t}e^{-(t-s)(-\triangle)^{\alpha}}h(s,x) ds\right\|_{L^{q}_{t}([0,\infty); L^{p}_{x})}
 \lesssim\|h\|_{L^{q_{1}'}_{t}([0,\infty); L^{p_{1}'}_{x})}.
$$
 Applying Proposition \ref{propostion} for $T=\infty$ and the Banach contraction mapping principle, we can prove (b) since
 $\|g\|_{L^{\frac{n}{2\alpha-1}}}+\|h\|_{L^{q_{1}'}_{t}([0,\infty); L^{p_{1}'}_{x})}$ is small enough.

 \vspace{0.1in}
 {{\it{3.11 Proof of Corollary  \ref{proposition}.}} The proof is similar to that of Proposition  \ref{10}.  We only demonstrate the
 case $|j|=1,$ since  similar arguments apply to the cases $|j|=2,3,\cdots, |k|.$ Define
\begin{equation}\label{operator2}
\overline{{T}}(Dv)=e^{-t(-\triangle)^{\alpha}}(Dg)+\int_{0}^{t}e^{-(t-s)(-\triangle)^{\alpha}}P(Dh)-B(Dv,v)-B(v,Dv).
\end{equation}
 Consider the integral equation $Dv=\overline{T}(Dv).$ Then $\overline{T}$ is a mapping of the space $X$ of function $v$ with
$$
v\in L^{q}([0,T];L^{p})\ \hbox{and}\ Dv\in L^{q}([0,T]; L^{p}).
$$
 The norm in $X$ is defined by
$$
  \|v\|_{X}=\|v\|_{L^{q}([0,T];L^{p})}+\|Dv\|_{L^{q}([0,T];L^{p})}.
 $$
 The assumption on $Dg$ and $Dh$ implies that the first two terms in the right hand side of (\ref{operator2}) are bounded in $X.$ The
 boundness of the other terms follows from Proposition\ref{propostion}. So,   $\overline{T}$ is a contraction mapping of $X$ into
 itself and  has a unique fixed point in $X.$ Therefore, the solution $v$ established in Proposition  \ref{10}
 satisfies $Dv\in L^{q}([0,T];L^{p}).$

 \vspace{0.1in}
\noindent
{\bf{Acknowledgements.}}

  This work is a part of my doctoral thesis, I want to thank my supervisor Professor  Jie Xiao for suggesting the
 problem  and for all  helpful discussions and  kind encouragement.

 \vspace{0.1in}

\bibliographystyle{amsplain}

\end{document}